\documentclass{amsart}
\usepackage{amssymb}
\usepackage{epsfig}
\usepackage{fancyhdr}

\makeatletter
\@namedef{subjclassname@2010}{%
  \textup{2010} Mathematics Subject Classification}
\makeatother

\def\Div{{\rm div\,}}
\def\rot{{\rm rot}\,}
\def\T{\rm T}
\def\D{\rm D}
\def\R{\rm R}

\newcommand{\al}{\alpha}
\newcommand{\be}{\beta}
\newcommand{\ga}{\gamma}
\newcommand{\de}{\delta}

\newcommand{\eps}{\varepsilon}
\newcommand{\va}{\varphi}
\def\la{\lambda}
\def\La{\Lambda}
\def\si{\sigma}
\def\De{\Delta}

\def\nb{\nabla}
\def\eps{\varepsilon}
\def\Om{\Omega}
\def\pa{\partial}
\def\Oi{\int_{\Omega}}

\def\v{\tilde{v}}

\def\ra{\rightarrow}

\newcommand{\benn}{\begin{eqnarray*}}
\newcommand{\eenn}{\end{eqnarray*}}
\newcommand{\ben}{\begin{eqnarray}}
\newcommand{\een}{\end{eqnarray}}
\newcommand{\bal}{\begin{aligned}}
\newcommand{\eal}{\end{aligned}}

\theoremstyle{plain}

\newtheorem{lemma}{Lemma}
\newtheorem{theorem}{Theorem}
\newtheorem{condition}{Condition}

\newtheorem{remark}[lemma]{Remark}

\newtheorem{definition}[lemma]{Definition}

\def\dis{\displaystyle}

\def\\{\hfil\break}
\def\const{{\rm const}}
\def\D{{\Bbb D}}
\def\I{{\Bbb I}}

\def\R{{\Bbb R}}
\def\T{{\Bbb T}}
\def\Z{{\Bbb Z}}

\def\d{\tilde{d}}
\def\h{\tilde{h}}
\def\ci{\tilde{\chi}}

\title[Nonstationary NSE]{Nonstationary flow for the Navier-Stokes equations in a cylindrical pipe }

\author[J. Renc{\l}awowicz \& W. M.
Zaj\c{a}czkowski]{Joanna Renc{\l}awowicz \& Wojciech M.
Zaj\c{a}czkowski}
\address{Joanna Renc{\l}awowicz:  Institute of Mathematics, Polish Academy of
Sciences, \'{S}nia\-dec\-kich 8, 00-956 Warsaw, Poland, e-mail:
jr@impan.gov.pl}
\address{Wojciech M. Zaj\c{a}czkowski:  Institute of Mathematics, Polish Academy of
Sciences, \'{S}niadeckich 8, 00-956 Warsaw, Poland, e-mail:
wz@impan.gov.pl and Institute of Mathematics and Cryptology,
Military University of Technology, Kaliskiego 2, 00-908 Warsaw,
Poland}



\thanks{Research supported by MNiSW grant no N N201 396937}
\date{\today}

\begin{document}

\begin{abstract}
In cylindrical domain, we consider the nonstationary flow with
prescribed inflow and outflow, modelled with Navier-Stokes
equations under the slip boundary conditions. Using smallness of
some derivatives of inflow function, external force and initial
velocity of the flow, but with no smallness restrictions on the
inflow, initial velocity neither force, we prove existence of
solutions in $W^{2,1}_2.$
\end{abstract}

\maketitle

\section{Introduction}

In the paper we study viscous incompressible fluid motion in a
finite cylinder with large inflow and outflow. The nonstationary
flow is described by Navier-Stokes equations and we impose the
boundary slip conditions, with friction between the flux and the
part of the boundary. Namely, we consider the following initial
boundary value problem.
\ben \label{NS} \bal %
&v_{t}+v\cdot\nabla v-\Div \T(v,p)=f\quad &{\rm in}\ \
 \Omega^T=\Omega\times(0,T),\\
&\Div v=0\quad &{\rm in}\ \ \Omega^T,\\
 &v\cdot\bar n=0\quad &{\rm
on}\  S_1^T ,\\
&\nu \bar n\cdot\D(v)\cdot\bar\tau_\alpha + \ga v \cdot
\bar\tau_{\al}=0,\ \ \alpha=1,2,\quad
 &{\rm on}\ \ S_1^T,\\
 &v\cdot\bar n=d\quad &{\rm
on}\  S_2^T ,\\
 & \bar n\cdot\D(v)\cdot\bar\tau_\alpha =0,\ \
\alpha=1,2,\quad
 &{\rm on}\ \ S_2^T,\\
&v\big|_{t=0}=v(0)\quad &{\rm in}\ \ \Omega, \eal \een \noindent
where $\Omega\subset\R^3$ is a cylindrical domain,
$S=\partial\Omega$, $v$ is the velocity of the fluid motion with
\mbox{$v(x,t)=(v_1(x,t),v_2(x,t),v_3(x,t))\in\R^3$} ,
$p=p(x,t)\in\R^1$ denotes the pressure,
$f=f(x,t)=(f_1(x,t),f_2(x,t),f_3(x,t))\in\R^3$ -- the external
force field, $x=(x_1, x_2, x_3)$ are the Cartesian coordinates,
$\bar n$ is the unit outward vector normal to the boundary $S$ and
$\bar\tau_\alpha$, $\alpha=1,2,$ are tangent vectors to $S$ and
$\cdot$ denotes the scalar product in $\R^3$.
 $\T(v,p)$ is the stress tensor of the form
$$
\T(v,p)=\nu\D(v)-p\I,
$$
where $\nu$ is the constant viscosity coefficient and  $\I$ is the
unit matrix. Next, $\ga >0$ is the slip coefficient and $\D(v)$
denotes the dilatation tensor of the form
$$
\D(v)=\{v_{i,x_j}+v_{j,x_i}\}_{i,j=1,2,3}.$$

We consider $\Omega\subset\R^3$ which is a~cylindrical type domain
parallel to the axis $x_3$ with arbitrary cross section. We split
$S$ into $S_1$ and $S_2$  so that $S = S_1 \cup S_2$ where $S_1$
is the part of the boundary which is parallel to the axis $x_3$
and $S_2$ is perpendicular to $x_3$.
Consequently, \benn S_1 & = & \{x\in\R^3:\varphi_0(x_1,x_2)=c_0,\ -a<x_3<a\}, \\
S_2(-a)& = & \{x\in\R^3:\varphi_0(x_1,x_2)<c_0,\ \ x_3 = -a \}, \\
S_2(a)& = & \{x\in\R^3:\varphi_0(x_1,x_2)<c_0,\ \ x_3 = a \} \eenn
where $a,c_0$ are positive given numbers and
$\varphi_0(x_1,x_2)=c_0$ describes a~sufficiently smooth closed
curve in the plane $x_3=\const.$

\begin{figure}[hbt]
\begin{center}
\epsfig{file=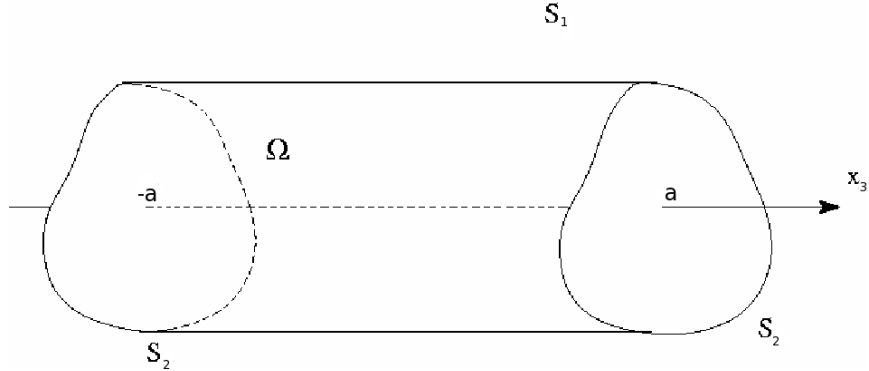,width=4.5 in}
\end{center}
\caption{Domain $\Omega.$}%
\label{om1}
\end{figure}

To describe inflow and outflow we define \ben \label{d0} \bal d_1
& =  -v\cdot \bar{n}|_{S_2(-a)} \\ d_2 & =  v \cdot
\bar{n}|_{S_2(a)} \eal \een with $d_i \ge 0, i=1,2.$ Using
$(\ref{NS})_{2,3}$ and (\ref{d0}) we conclude the following
compatibility condition \ben \label{d} \Phi \equiv \int_{S_2(-a)}
d_1dS_2 = \int_{S_2(a)} d_2 dS_2, \een where $\Phi$ is the flux.

The main goal of the paper is to show the long time existence of
regular solutions to problem (\ref{NS}) with arbitrary large
external force $f$, initial velocity $v(0)$ and inflow-outflow
$d.$ We can obtain the existence result by regularizing weak
solutions. In general, we follow the ideas from \cite{RZ1,Z1}.

Regularity of solutions for Navier-Stokes equations, even with no
flux, requires some smallness conditions. Many results of this
type were proved assuming smallness of initial velocity, some
restrictions on domain (so called thin domain, $\Om = \Om' \times
(0,\eps), \Om' \in \R^2$ with small $\eps$) or special structure
of solutions (so that the solution is close to 2-dimensional
solution). We mention here as well some results in case of
stationary or nonstationary flux, with the same type of boundary
condition that we assume: slip condition on the part of boundary,
involving or neglecting friction between the boundary and the
fluid. In \cite{M1}, the existence for large data is obtained
under some geometrical constraints for 2d model in steady and
evolutionary case. In \cite{M2}, steady Navier Stokes equations in
pipe-like domain are investigated and existence is shown for a
class of cylindrical symmetric solutions. In \cite{Z4}, the
problem of nonstationary flow in axially symmetric domain is
examined and the result concerns the existence for solutions close
to axially symmetric solutions and the inflow and outflow
sufficiently close to homogeneous flux.

We show the existence of regular solutions in a class of flows
with arbitrary large initial velocity $v(0)$, inflow $d_1$ and
external force $f.$ Instead, we use smallness of derivatives of
initial velocity and force in the direction along the $x_3$ axis
and smallness of some derivatives of the inflow.

To formulate the main result, we introduce some notation.

We will use Sobolev spaces
$$
W_p^{s,s/2}(Q^T),\quad Q\in\{\Omega,S\},\quad
s\in\Z_+\cup\{0\},\quad p\in[1,\infty],
$$
with the following norm for even $s$ \benn W_p^{s,s/2}(Q^T)= \{u:
\|u\|_{W_p^{s,s/2}(Q^T)}=\bigg(\sum_{|\alpha|+2a\le s}\int_{Q^T}
|D_x^\alpha\partial_t^a u|^p dxdt\bigg)^{1/p} \le \infty \},
\eenn where
$D_x^\alpha=\partial_{x_1}^{\alpha_1}\partial_{x_2}^{\alpha_2}
\partial_{x_3}^{\alpha_3}$, $\alpha=(\alpha_1,\alpha_2,\alpha_3)$,
$|\alpha|=\alpha_1+\alpha_2+\alpha_3$, $\alpha_i\in\Z_+\cup\{0\}$,
$i=1,2,3.$ %
In particular, $$ H^s(Q)=W_2^s(Q). $$ We will need also%
\benn V_2^0(\Om^T) = \{u: ||u||_{V^0_2(\Om^T)} = {\rm ess}
\sup_{t\in(0,T)}||u||_{L_2(\Om)} + \left(\int_0^T
||\nb u||_{L_2(\Om)}^2 dt \right)^{1/2} <\infty \}, \\
V_2^1(\Om^T) = \{u: ||u||_{V^1_2(\Om^T)} = {\rm ess}
\sup_{t\in(0,T)}||u||_{H^1(\Om)} + \left(\int_0^T ||\nb
u||_{H^1(\Om)}^2 dt \right)^{1/2} <\infty \}. \eenn

Let $A$ be a bound for the $V_0^2$ norm of $v$ -- a weak solution
of (\ref{NS}), i.e. \benn \|v\|^2_{V^0_2(\Om^t)} + \ga
\sum_{\al=1}^2 \int_0^t \|v\cdot\bar{\tau}_{\al}\|_{L_2(S_1)}^2
\le A^2 \eenn where the existence of such constant can be obtained
as it is shown in \cite{RZ1}, see Lemma~\ref{def-A} below.

\begin{definition}
We define \ben \label{D0}
 D_0 =\|d\|^6_{L_{\infty}(0,t;L_{3}(S_2))} +A^2 +1, \een
\ben \bal \label{la3-4} \La_2 = \|(\rot f)_3\|_{L_2(0,t;
H^2(\Om))}+ \|(\rot v)_3(0)\|_{L_2(\Om)}+ 1
\\ \La_3 =  \|f\|_{L_{5/3}(\Om^T)} +
\|v(0)\|_{W_{5/3}^{4/5}(\Om)}, \\ \La_4 = \|f\|_{L_2(\Om^T)} +
\|v(0)\|_{H^1(\Om)} .  \eal \een%
We note $h=v,_{x_3}, g=f,_{x_3}, x'=(x_1, x_2).$ We set \ben \bal \label{La-def}
\La = D_1+ D_2 +\|f_3\|^2_{L_2(0,t;L_{4/3}(S_2))} +
\|g\|^2_{L_2(0,t;L_{6/5}(\Om))}+ \|h(0)\|^2_{L_2(\Om)},\eal \een
where \ben \label{D-def} \bal
 D_1 & =  \|d_t\|^2_{L_2(0,t;H^1(S_2))} + \|d_{,x'}\|^2_{L_2(0,t;H^1(S_2))}, \\
D_2 & =  \|d_{,x'}\|^2_{L_{\infty}(0,t;H^1(S_2))}. \eal \een
\end{definition}

\begin{condition}
Let $D_0, \La, \La_2, \La_3, \La_4$ be finite.
\end{condition}

\begin{condition}
Let $\La$ be so small and $\|g\|_{L_2(\Om^T)},$
$\|h(0)\|_{H^1(\Om)}$ be such that there exists constant
$\mathcal{A}$ satisfying the relation \benn   c D_0^2 \La^2
\left[({\mathcal A} + \La) +
c({\mathcal A} + {\mathcal A}^2) + c \La_2 + c\La_3 \right] \\
\exp \left[\int_0^t (\|d\|^6_{L_3(S_2)}dt + cD_0({\mathcal A} +
\La) + c({\mathcal A} + {\mathcal A}^2) + c \La_2 + c\La_3 \right]
\\ + \|g\|^2_{L_2(\Om^T)} + \|h(0)\|^2_{H^1(\Om)}\le {\mathcal A}^2  \eenn

\end{condition}

\begin{theorem}
Assume that Conditions 1 and 2 hold. Then there exists a~solution
to problem (\ref{NS}) such that
$$
v,v_{,x_3}\in W_2^{2,1}(\Om^{T}),\nb p,\nb p_{,x_3}\in
L_2(\Om^{T})
$$
and \ben \label{main} \bal
&\|v_{,x_3}\|_{W_2^{2,1}(\Om^{T}}\le \mathcal{A},\\
&\|v\|_{W_2^{2,1}(\Om^{T}}+\|\nabla p\|_{L_2(\Om^{T})}\le
\va(\mathcal{A},D_0, \La, \La_2, \La_3, \La_4),\\
&\|\nb p_{,x_3}\|_{L_2(\Om^{T})}\le\va(\mathcal{A},D_0, \La, \La_2, \La_3, \La_4), \\
\eal \een where $\va$ is an increasing positive function.
\end{theorem}

The structure of the paper is the following. First, in Section 2,
since the boundary conditions in the problem (\ref{NS}) are not
homogeneous, we introduce auxiliary variables and study the
corresponding equations. In Section 3.1 and 3.2 we obtain the
energy inequalities for $h=v,_{x_3}$ and next for $\chi= (\rot
v)_{3}$. Then, through the elliptic system relating $\chi$ and
$v_{x_1}, v_{x_2}$, we can conclude the relation
for $v$ in terms of $h$ norms in Section 3.3. Next, the energy inequality for $h$ is
improved by getting a bound for the energy norm of $h,$ depending on some norm of $\nb v.$
This relation is used to show the estimate for $h$ and $v$. In
Section 4, we prove the existence using a priori estimates and
the Leray-Schauder theorem.

\section{Auxiliary problems}
\setcounter{equation}{0}%
\setcounter{lemma}{0}

 In order to define
weak solutions to problem (\ref{NS}) and to obtain the energy type
estimate we need to make the boundary condition $(\ref{NS})_4$
homogeneous. Thus, we have to reformulate the main problem using
some auxiliary functions.

First, we extend functions corresponding to inflow and
outflow so that \ben \d_i|_{S_2(a_i)} = d_i,\,\, i= 1,2,\, a_1 =
-a,\, a_2= a \een We introduce the function $\eta$, see \cite{L}.
\benn \eta(\si;\eps,\rho) = \left\{ \begin{array}{lr} 1 & 0 \le
\si \le \rho e^{-1/\eps}\equiv r, \\ -\eps
\ln\dis{\frac{\si}{\rho}} & r < \si\le \rho, \\ 0 & \rho < \si <
\infty.
\end{array} \right.
\eenn

We calculate \benn \frac{d\eta}{d\si} = \eta'(\si;\eps,\rho) =
\left\{
\begin{array}{lr} 0 & 0 < \si \le r, \\
-\dis{\frac{\eps}{\si}} & r < \si\le \rho, \\ 0 & \rho < \si <
\infty.
\end{array} \right.
\eenn so that $|\eta'(\si;\eps,\rho)| \le \dis{\frac{\eps}{\si}}.$
We define functions $\eta_i$ on the neighborhood of $S_2$ (inside
$\Om$): \benn \eta_i= \eta(\si_i;\eps,\rho),\  i=1,2,\eenn where
$\si_i$ denote local coordinates defined on small neighborhood of
$S_2(a_i):$ \benn \si_1 = a+x_3, \ \si_2= a-x_3 \eenn and we set
\ben \label{al} \bal \al & = \sum_{i=1}^2 \d_i \eta_i,
\\ b & = \al \bar{e}_3, \ \bar{e}_3= (0,0,1). \eal \een
We construct function $u$ so that \ben \label{u} u= v-b. \een
Therefore, \benn \Div u & = & - \Div b = -\al_{x_3}\quad {\rm in}
\ \ \Om, \\ u\cdot \bar{n} & = & 0 \quad {\rm on} \ \ S. \eenn
Then, the boundary condition for $u$ is homogeneous. The
compatibility condition takes the form \benn \int_{\Om} \al,_{x_3}
dx = -\int_{S_2(-a)} \al|_{x_3=-a} dS_2 + \int_{S_2(a)}
\al|_{x_3=a}dS_2 = 0 \eenn We define function $\va$ as a solution
to the Neumann problem \ben \bal \label{va} \De \va & = - \Div b
\quad {\rm in} \ \ \Om, \\ \bar{n}\cdot\nb \va & = 0 \quad {\rm
on} \ \ S,\\ \int_{\Om} \va dx & = 0. \eal \een Next, we set \ben
\label{w} w = u - \nb \va = v - (b+\nb \va) \equiv v - \de. \een

Then, $(w,p)$ is a solution to the following problem \ben
\label{NS-w} \bal w_{t}+w\cdot\nabla w + w\cdot\nabla \de + \de
\cdot\nabla w -\Div \T(w,p) & \\  =  f - \de_t -\de \cdot\nabla
\de + \nu\Div \D(\de) = F(\de,t) \quad &{\rm in}\ \
 \Omega^T,\\
\Div w=0\quad &{\rm in}\ \ \Omega^T,\\
 w\cdot\bar n=0\quad &{\rm
on}\  S^T ,\\
\nu \bar n\cdot\D(w)\cdot\bar\tau_\alpha + \ga w \cdot
\bar\tau_{\al} & \\ = - \nu \bar n\cdot\D(\de)\cdot\bar\tau_\alpha
- \ga \de \cdot \bar\tau_{\al} = B_{1\al}(\de) ,\ \
\alpha=1,2,\quad
 &{\rm on}\ \ S_1^T,\\
  \bar n\cdot\D(w)\cdot\bar\tau_\alpha = -\bar n\cdot\D(\de)\cdot\bar\tau_\alpha= B_{2\al}(\de),\ \
\alpha=1,2,\quad
 &{\rm on}\ \ S_2^T,\\
w\big|_{t=0}=v(0) - \de(0) = w(0)\quad &{\rm in}\ \ \Omega, \eal
\een where $\Div \de =0.$ Moreover, we set \benn \bar{n}|_{S_1} =
\frac{(\va_0,_{x_1}, \va_0,_{x_2}, 0)}{\sqrt{\va_0,_{x_1}^2 +
\va_0,_{x_2}^2}},\  \bar{\tau}_1|_{S_1}= \frac{(-\va_0,_{x_2},
\va_0,_{x_1}, 0)}{\sqrt{\va_0,_{x_1}^2 + \va_0,_{x_2}^2}},\
\bar{\tau}_2|_{S_1}=(0,0,1) = \bar{e}_3, \\ \bar{n}|_{S_2(-a)} = -
\bar{e}_3, \ \bar{n}|_{S_2(a)} = \bar{e}_3,\
\bar{\tau}_1|_{S_2}=\bar{e}_1,\  \bar{\tau}_2|_{S_2}= \bar{e}_2
\eenn where $\bar{e}_1=(1,0,0), \bar{e}_2= (0,1,0).$
 For the system (\ref{NS-w}) we can define weak solutions
and then from \cite{RZ1} we have
\begin{lemma} \label{def-A}
Assume the compatibility condition (\ref{d}). Assume that $v(0)
\in L_2(\Om),$ $f \in L_2(0,T;L_{6/5}(\Om)),$ $d_i \in
L_{\infty}(0,T; W^{s-1/p}_p(S_2)) \bigcap
L_2(0,T;W^{1/2}_2(S_2)),$ $\frac{3}{p} +\frac{1}{3} \le s, p>3$ or
$p=3, s> \frac{4}{3},$ $d_{i,t} \in L_2(0,T;W^{1/6}_{6/5}(S_2)),
i=1,2.$
 Then there exists a weak solution $v$ to problem (\ref{NS}) such
that $v$ is weakly continuous with respect to $t$ in $L^2(\Om)$ norm and
$v$ converges to $v_0$ as $t\ra 0$ strongly in $L^2(\Om)$ norm.
Moreover, $v\in V^0_2(\Om^T),$ $v\cdot\bar{\tau}_{\al} \in
L_2(0,T;L_2(S_1)), \al =1,2,$ and $v$ satisfies \ben \label{est-v}
\bal & \|v\|^2_{V^0_2(\Om^t)} + \ga \sum_{\al=1}^2 \int_0^t
\|v\cdot\bar{\tau}_{\al}\|_{L_2(S_1)}^2 \le 2
\|f\|^2_{L_2(0,t;L_{6/5}(\Om))}
\\ & + \va\left(\sup_{\tau\le t}\|d\|_{W^{s-1/p}_3(S_2)}\right)\left(\|d\|_{L_2(0,t;W^{1/2}_2(S_2))}^2 +
\|d_t\|_{L_2(0,t;W^{1/6}_{6/5}(S_2))}^2\right) +
\|v(0)\|_{L_2(\Om)}^2 \equiv A^2 \eal \een where $\va$ is a nonlinear
positive increasing function of its argument and $t\le T.$
\end{lemma}
Since we would like to have not restricted magnitudes of $v(0), f$
and $d,$ we would need some smallness of $v_{x_3}(0)$ and
$f_{x_3}$ in $L_2$ norms, instead. Therefore, we set $h=v,_{x_3},
\ q=p_{,x_3}$ that are solutions to the problem: (see Lemma~3.1,
\cite{Z1}.) \ben \label{h-q} \bal h_{,t}-\Div \T(h,q)& =
-v\cdot\nabla h-h\cdot
 \nabla v+g\quad &{\rm in}\ \ \Omega^T,\\
\Div h &= 0\quad &{\rm in}\ \ \Omega^T,\\ \bar{n}\cdot h &= 0,\ \
\bar n\cdot\D(h)\cdot\bar\tau_\alpha + \ga h \cdot\bar\tau_\alpha
=0,\ \
 \alpha=1,2\quad &{\rm on}\ \ S_1^T,\\
h_i &= -d_{x_i},\ \ i=1,2,\ \ h_{3,x_3}= \De' d \quad &{\rm on}\ \ S_2^T,\\
h\big|_{t=0}&= h(0)\quad &{\rm in}\ \ \Omega.\eal \een where $g=
f_{,x_3}, \De' = \pa_{x_1}^2 + \pa_{x_2}^2.$ With $\chi=(\rot
v)_3$, we consider the problem \ben \label{v-chi} \bal
&v_{1,x_2}-v_{2,x_1}=\chi\quad &{\rm in}\ \ \Omega',\\
&v_{1,x_1}+v_{2,x_2}=-h_3 \quad &{\rm in}\ \ \Omega',\\
&v'\cdot\bar n'=0\quad &{\rm on}\ \ S'_1, \eal \een where $\Om'$ is a cross-section of $\Om$ by the plane perpendicular to the $x_3$ axis crossing it within the interval $(-a,a)$ and $\pa \Om' = S_1 \bigcap P \equiv S_1'.$ Namely, \benn \Om'& = & \Om \cap \{ {\rm plane:} \  x_3={\rm const} \in
(-a,a) \} \\
S'_1 & = & S_1\cap\, \{{\rm plane:}\ x_3 = \const\in(-a,a)\} \eenn
and $x_3$, $t$ are treated as parameters. Then, the function
$\chi=(\rot v)_3= v_{2,x_1} - v_{1,x_2}$ satisfies (see Lemma~3.2
in \cite{Z1}) \ben \label{chi} \bal
&\chi_{,t}+v\cdot\nabla\chi-h_3\chi+h_2v_{3,x_1}-h_1v_{3,x_2}
 -\nu\Delta\chi=F_3\quad &{\rm in}\ \ \Omega^T,\\
&\chi=v_i(n_{i,x_j}\tau_{1j}+\tau_{1i,x_j}n_j)+v\cdot\bar\tau_1
 (\tau_{12,x_1}-\tau_{11,x_2})\\ & + \frac{\ga}{\nu} v_j \tau_{1j}
\equiv\chi_*\quad &{\rm on}\ \ S_1^T,\\  &\chi_{,x_3}=0\quad
&{\rm on}\ \ S_2^T,\\ &\chi\big|_{t=0}=\chi(0)\quad &{\rm in}\ \
\Omega, \eal \een%
where $F_3 = f_{2,x_1} - f_{1,x_2}.$

\section{Estimates}
\setcounter{equation}{0} %
\setcounter{lemma}{0}

In order to obtain the energy estimate for $h$ -- a solution to
problem (\ref{h-q}) we have to make the Dirichlet boundary
condition on $S_2^T$ homogeneous. For this purpose we are looking
for such a function $\tilde{h}$ that \ben \label{-h} \bal
\Div \h & = 0 \quad {\rm in} \ \ \Om, \\
\h & = 0 \quad {\rm on} \ \ S_1, \\
\h_i & = -d_{,x_i}, \ i=1,2 \quad {\rm on} \ \ S_2, \\
\h_3 & = 0 \quad {\rm on} \ \ S_2. \eal \een

\begin{lemma} \label{h-fala} There exists a function $\h$ satisfying (\ref{-h}) such that \ben \label{h-} \bal
\| \h\|_{W^1_{\si}(\Om)} \le c(\eps + \rho^a)\|d_{,x'} \|_{W^1_{\si}(S_2)} , \\
\|\h_{,t}\|_{L_{\si}(\Om)}\le c (\eps + \rho^a)\|d_{,x't} \|_{L_{\si}(S_2)}, \eal \een
where $\si \in (1,\infty)$ and $a>0.$
\end{lemma}

\begin{proof}
First, we construct function $\h.$ Let us define the functions
\benn
\bar{h}_i & = & -(\eta_1 \d_{1,x_i} +\eta_2\d_{2,x_i}), i=1,2, \\
\bar{h}_3 & = & 0, \eenn where $\eta_1, \eta_2$ are smooth cut-off
functions introduced before and $\d_i$ is an extension of $d_i$ on
a neighbourhood of $S_2(a_i), i=1,2.$ We have the compatibility
condition \ben \label{comp} \sum_{i=1}^2 n_i|_{S_1} \cdot
d_{\al,x_i} = 0, \quad \al=1,2 \een which follows from the fact
that $\bar{n}|_{S_1}$ does not depend on $x_3,$ so \benn v\cdot
\bar{n}|_{S_1} = 0 \Rightarrow v,_{x_3}\cdot\  \bar{n}|_{S_1} = 0
\Rightarrow h\cdot \bar{n}|_{S_1} = 0 \Rightarrow h|_{S_2}\cdot
\bar{n}|_{S_1} = 0 \eenn and in view of $(\ref{h-q})_4$ the
relation (\ref{comp}) holds.

Hence $\bar{h}$ is a solution to the problem \benn \Div \bar{h} =
- (\eta_1 \De' \d_{1} +\eta_2 \De'\d_{2}) \quad {\rm in}\ \ \Om, \\
\bar{h} \cdot {\tau_1} = - (\eta_1 \d_{1,{x_i}} +\eta_2
\d_{2,{x_i}}) {\tau_{1i}} \quad {\rm on}\ \ S_1, \\ \bar{h}
\cdot {\tau_2} = 0 \quad {\rm on}\ \ S_1,
\\ \bar{h} \cdot \bar{n} = 0 \quad {\rm on}\ \ S_1, \\
\bar{h}_i = -d_{j,x_i}, \quad i,j=1,2 \quad {\rm on} \ \ S_2(a_j), \\
\bar{h}_3 = 0 \quad {\rm on} \ \ S_2. \eenn Next we define a
function $\phi$ such that \benn \De \phi = -(\eta_1 \De' \d_{1}
+\eta_2 \De' \d_{2}), \\ \bar{n} \cdot \nb\phi |_S = 0, \eenn and
we are looking for a function $\la$ and $\si$  such that \ben
\label{la}
\bal -\De \la+\nb \si = 0, \\ \Div \la =0, \\
\la\cdot\bar{\tau}_{\be}
= -\bar{\tau}_{\be}\cdot \nb \phi + \bar{h}\cdot \bar{\tau}_{\be}, \quad \be=1,2, \ \ {\rm on}\  S_1 \\
\la\cdot \bar{n} = 0, \ \ {\rm on} \ S_1 \\
\la_i= -\nb_i \phi,\ \ i=1,2, \ \ {\rm on}\  S_2 \\
\la_3 = 0 \ \ {\rm on}\  S_2. \eal \een Then \benn \h=\bar{h} -
(\la+\nb \phi) \eenn is a solution to problem (\ref{-h}). We can
get some estimates for $\h.$ We note that \benn
\|\bar{h}\|_{L_p(\Om)} \le \sum_{j=1}^2
\|\eta_j\d_{j,x'}\|_{L_p(\Om)} \le c\rho^{1/p'_0}
\sum_{j=1}^2\|\eta_j\d_{j,x'}\|_{L_{p_0}(\Om)}, \quad {\rm where}
\ \ \frac{1}{p_0}+\frac{1}{p'_0}=\frac{1}{p}. \eenn Let $G$ be the
Green function to the Neumann problem for $\phi.$ Then we have
\benn \phi(x) = \int_{\Om} G(x,y) \pa_{y_i}(\eta_1\d_{1,y_i} +
\eta_2\d_{2,y_i}) dy = -\Oi \nb_{y_i} G(x,y)(\eta_1\d_{1,y_i} +
\eta_2\d_{2,y_i}) dy,\eenn where we used that $\bar{h} \cdot
\bar{n}|_{S_1} = 0.$  Then we calculate \benn \nb_x \phi(x) = -\Oi
\nb_x \nb_{y_i} G(x,y)(\eta_1\d_{1,y_i} + \eta_2\d_{2,y_i})
dy\eenn and \ben \bal \label{green} \|\nb_x\phi\|_{L_p(\Om)} \le
\sum_{j=1}^2 \|\eta_j\d_{j,x'}\|_{L_p(\Om)} \le c\rho^{1/p'_0}
\sum_{j=1}^2\|\eta_j\d_{j,x'}\|_{L_{p_0}(\Om)}, \\ {\rm where} \ \
\frac{1}{p_0}+\frac{1}{p'_0}=\frac{1}{p}. \eal \een Applying now
the existence of the Green function to problem (\ref{la}) we have
\benn \la_i(x) = \int_{S_1} \frac{\pa G_{i\be}}{\pa n_{S_1}}
(-\bar{\tau}_{\be}\cdot \nb \phi + \bar{h}\cdot
\bar{\tau}_{\be})dS_1 + \int_{S_2}\frac{\pa G_{ij}}{\pa n_{S_2}}
(-\nb_j \phi) dS_2. \eenn In view of (\ref{green}) we obtain \ben
\label{lam} \|\la\|_{L_p(\Om)} \le c\rho^{1/p'_0}
\sum_{j=1}^2\|\eta_j\d_{j,x'}\|_{L_{p_0}(\Om)}. \een Continuing
the considerations and using properties of functions $\eta_i,
i=1,2$, yields that constructed function $\h$ satisfies
(\ref{h-}). This concludes the proof.
\end{proof}

\subsection{A priori inequalities for functions $k$ and $h$}

Let us introduce the new function $$ k= h-\h.$$ Then $k$ is a
solution to the problem \ben \label{k} \bal k_{,t}-\Div \T(h,q)& =
-v\cdot\nabla h-h\cdot
 \nabla v - \h,_t+g \equiv \bf{g} \quad &{\rm in}\ \ \Omega^T,\\
\Div k &= 0\quad &{\rm in}\ \ \Omega^T,\\ \bar{n}\cdot k &= 0,\ \
\nu \bar n\cdot\D(h)\cdot\bar\tau_\alpha + \ga h
\cdot\bar\tau_\alpha =0,\ \
 \alpha=1,2\quad &{\rm on}\ \ S_1^T,\\
k_i &= 0,\ \ i=1,2,\ \ h_{3,x_3}= \De' d \quad &{\rm on}\ \ S_2^T,\\
k\big|_{t=0}&= h(0)-\h(0) \equiv k(0)\quad &{\rm in}\ \
\Omega.\eal \een where $g= f_{,x_3}, \De' = \pa_{x_1}^2 +
\pa_{x_2}^2.$%
We have the following Korn inequality
\begin{lemma}\cite{Z1}
Assume that $\Om$ is not axially symmetric, \benn E_{\Om}(k) =
\sum_{i,j=1}^3 \Oi(k_{i,x_j} +k_{j,x_i})^2 dx,\ \  E_{\Om}(k) +
\sum_{{\al}=1}^2 \|k\cdot \bar{\tau}_{\al}\|^2_{L_2(S_1)} <\infty, \\
\|\De'd\|_{L_2(S_2)} + \|\h_{3,x_3}\|_{L_2(\Om)} +
\|\h_3\|_{L_2(\Om)} <\infty. \eenn Then the following inequality
is valid \ben \label{k-h1} \bal \|k\|^2_{H^1(\Om)} \le
c\left(E_{\Om}(k) + \sum_{{\al}=1}^2\ga\|k\cdot
\bar{\tau}_{\al}\|^2_{L_2(S_1)}\right) \\ +c
(\|\De'd\|_{L_2(S_2)}^2 + \|\h_{3,x_3}\|_{L_2(S_2)}^2 +
\|\h_3\|_{L_2(\Om)}). \eal \een
\end{lemma}

Let us consider the problem (\ref{k}) in the form \ben \label{k1}
\bal k_{,t}-\Div \T(h,q)& = -v\cdot\nabla k-k\cdot
 \nabla v -v\cdot\nabla \tilde{h} -\tilde{h}\cdot
 \nabla v - \h,_t+g  \quad &{\rm in}\ \ \Omega^T,\\
\Div k &= 0\quad &{\rm in}\ \ \Omega^T,\\ \bar{n}\cdot k &= 0,\ \
\nu \bar n\cdot\D(h)\cdot\bar\tau_\alpha + \ga h
\cdot\bar\tau_\alpha =0,\ \
 \alpha=1,2\quad &{\rm on}\ \ S_1^T,\\
k_i &= 0,\ \ i=1,2,\ \ h_{3,x_3}= \De' d \quad &{\rm on}\ \ S_2^T,\\
k\big|_{t=0}&= h(0)-\h(0) \equiv k(0)\quad &{\rm in}\ \
\Omega.\eal \een Using function $w=u-\de$ we reformulate the $k$
problem into the following system \ben \label{k1-w} \bal
k_{,t}-\Div \T(h,q)& = -w\cdot\nabla k-k\cdot
 \nabla v &  \\  & - \de \cdot \nb k - v\cdot\nabla \tilde{h} -\tilde{h}\cdot
 \nabla v - \h,_t+g \equiv G \quad &{\rm in}\ \ \Omega^T,\\
\Div k &= 0\quad &{\rm in}\ \ \Omega^T,\\ \bar{n}\cdot k &= 0,\ \
\nu \bar n\cdot\D(h)\cdot\bar\tau_\alpha + \ga h
\cdot\bar\tau_\alpha =0,\ \
 \alpha=1,2\quad &{\rm on}\ \ S_1^T,\\
k_i &= 0,\ \ i=1,2,\ \ h_{3,x_3}= \De' d \quad &{\rm on}\ \ S_2^T,\\
k\big|_{t=0}&=  k(0)\quad &{\rm in}\ \ \Omega.\eal \een

Projecting $\Div k$ on $S_2$ we see that \benn \Div k|_{S_2} =
k_{3,x_3}|_{S_2} = 0 \eenn Then the second condition in
$(\ref{k1-w})_4$ takes the form \benn \h_{3,x_3} = \De' d \quad
{\rm on} \quad S_2 \eenn

\begin{lemma}
Let \benn  \La & = & D_1+ D_2 +\|f_3\|^2_{L_2(0,t;L_{4/3}(S_2))}
\\ & & +
\|g\|^2_{L_2(0,t;L_{6/5}(\Om))}+ \|h(0)\|^2_{L_2(\Om)}, \\
D_1 & = &  \|d_t\|^2_{L_2(0,t;H^1(S_2))} +
\|d_{,x'}\|^2_{L_2(0,t;H^1(S_2))} \\
D_2 & = &  \|d_{,x'}\|^2_{L_{\infty}(0,t;H^1(S_2))},
  \eenn
and \ben \label{H-def} H_1 = \|h\|^2_{L_2(\Om^t)}+
\|h\|^2_{L_{\infty}(0,t;L_3(\Om))}. \een Then the  solution $k$ to
the equivalent problems (\ref{k}, \ref{k1},\ref{k1-w}) satisfies
\ben \bal \label{k-Dcomp} \|k\|^2_{V_2^0(\Om^t)} +
\sum_{\al=1}^2\|k\cdot \bar{\tau}_{\al}\|^2_{L_2(S_1^t)} \\
\le c (\|d\|^6_{L_{\infty}(0,t;L_{3}(S_2))}+A^2+1)(\La^2 + H_1^2)
\eal \een
\end{lemma}
\begin{proof}
We shall obtain the energy type estimate for solutions to problem
(\ref{k1-w}). Multiplying $(\ref{k1-w})_1$ by $k$ and integrating
over $\Om$ yields \ben \label{k=} \bal \frac{1}{2}\frac{d}{dt}
\int_{\Om} k^2 dx - \Oi \Div \T(h,q) k dx  = -\Oi w\cdot\nb k k dx
- \Oi k\nb v\cdot k dx  \\ -\Oi \de\cdot\nb k k dx - \Oi v \cdot
\nb\h k dx - \Oi \h\cdot \nb v k dx - \Oi\h_t k dx +\Oi gk dx.
\eal \een Now we examine the particular terms in (\ref{k=}).
Integrating by parts the second term on the l.h.s takes the form
\benn -\int_{S_1} \bar{n}\cdot \T(h,q)\cdot k dS_1 - \int_{S_2}
\bar{n}\cdot \T(h,q)\cdot k dS_2 + \frac{\ga}{2} \Oi \D(h)\D(k) dx
\equiv I_1 +I_2+I_3,\eenn where \benn I_1 = -\int_{S_1}
\bar{n}\cdot \T(h,q)\cdot\bar{\tau}_{\al}k\cdot \bar{\tau}_{\al}
dS_1 = \ga\int_{S_1}h\cdot\bar{\tau}_{\al}k\cdot \bar{\tau}_{\al}
dS_1 \\ = \ga\|k\cdot\bar{\tau}_{\al}\|^2_{L_2(S_1)}+ \ga
\int_{S_1}\h\cdot\bar{\tau}_{\al}k\cdot \bar{\tau}_{\al} dS_1, \\
I_2 = -\int_{S_2} \T_{nn}(h,q) k_n dS_2 = -\int_{S_2}(2\nu
h_{3,x_3} - q) k_n dS_2 = -\int_{S_2}(2\nu \h_{3,x_3} - q) k_n
dS_2 \\ = -2\nu \int_{S_2} \h_{3,x_3}k_n dS_2 - \int_{S_2(-a)} q
k_3 dS_2 + \int_{S_2(a)} q k_3 dS_2 \eenn To examine the last
integral we use the third component of $(\ref{NS})_1$ projected on
$S_2$ \ben \label{proj} d_t + v\cdot \nb v_3 -\nu \De' d -
\nu\h_{3,x_3} -f_3= -q \een Using this relation in the last term
of $I_2$ we obtain \benn \int_{S_2} qk_3dS_2 = \int_{S_2}(-d_t
+\nu \De' d+\nu\h_{3,x_3} +f_3)k_3 dS_2 - \int_{S_2} v'\nb' d k_3 dS_2
\\
+ \int_{S_2(-a)} dk_3^2dS_2 -\int_{S_2(a)} dk_3^2dS_2 \eenn where
$$ d|_{S_2(-a)} = -d_1, d_1>0, d|_{S_2(a)} = d_2 >0$$ and  in the first two terms we do not distinguish dependence on
$S_2(-a)$ and $S_2(a)$ because it does not have any influence on
estimations.

The first term in the above expression we estimate by \benn
\eps_1\|k_3\|^2_{L_4(S_2)} +
c(1/\eps_1)\left(\|d_t\|^2_{L_{4/3}(S_2)} + \|\De'
d\|^2_{L_{4/3}(S_2)}+\|f_3\|^2_{L_{4/3}(S_2)}\right), \eenn the
second term by \benn \eps_2\|k_3\|^2_{L_4(S_2)} + c(1/\eps_2)
\|v'\|^2_{L_4(S_2)} \|\nb' d\|^2_{L_2(S_2)} \eenn and the last one
as follows \benn \int_{S_2(-a)} d_1 k_3^2 dS_2 \le
\|d\|_{L_3(S_2)} \|k_3\|^2_{L_3(S_2)} \le \left(\eps^{1/6}  \| \nb
k_3\|^2_{L_2(\Om)} + c\eps^{-5/6}\|k_3\|^2_{L_2(\Om)} \right)
\cdot \|d\|_{L_3(S_2)}\\ \le \eps_2 \|\nb k_3\|^2_{L_2(\Om)} +
c\eps_2^{-5/6} \|d\|^6_{L_3(S_2)} \cdot (\|h_3\|^2_{L_2(\Om)} +
\|\h_3\|^2_{L_2(\Om)}). \eenn

Using the above estimates of the second term on the l.h.s. of
(\ref{k=}) and denoting the r.h.s. of (\ref{k1-w}) by $G$ we
obtain \benn \bal \frac{1}{2}\frac{d}{dt} \Oi k^2 dx +
\frac{\nu}{2} \Oi |\D(k)|^2 dx +\frac{\ga}{2} \int_{S_1}
|k\cdot\bar{\tau}_{\al}|^2 dS_1  \le \frac{\nu}{4}
\|k\|^2_{H^1(\Om)} + \frac{\ga}{2}\int_{S_1} |\h\cdot\bar{\tau}_{\al}|^2 dS_1 \\
+ c \Oi |\D(\h)|^2 dx  + c(\|d_t\|^2_{L_{4/3}(S_2)} + \|\De'
d\|^2_{L_{4/3}(S_2)}+\|f_3\|^2_{L_{4/3}(S_2)}) \\ +
\|d\|^6_{L_3(S_2)} \cdot (\|h_3\|^2_{L_2(\Om)} +
\|\h_3\|^2_{L_2(\Om)})  + |\Oi G h dx | \eal \eenn We apply the
Korn inequality (\ref{k-h1}) to conclude \ben \label{k-korn} \bal
\frac{d}{dt} \Oi k^2 dx + {\nu} \|k\|_{H^1(\Om)}^2 +\ga
\|k\cdot\bar{\tau}_{\al}\|_{L_2(S_1)}^2  \\ \le
c(\|d_t\|^2_{L_{4/3}(S_2)} + \|\De'
d\|^2_{L_{4/3}(S_2)}+\|f_3\|^2_{L_{4/3}(S_2)} +
\|\h\|^2_{H^1(\Om)} )
\\+ \|d\|^6_{L_3(S_2)} \cdot (\|h_3\|^2_{L_2(\Om)} +
\|\h_3\|^2_{L_2(\Om)})  + |\Oi G h dx |
\eal \een %
Now we shall examine the last term on the r.h.s. of the
inequality (\ref{k-korn}). To this end we use the r.h.s. of
(\ref{k=}). The first term on the r.h.s. of (\ref{k=}) vanishes.
The second term can be estimated either by \benn \int_{\Om}k^2|\nb
v| dx \le \eps_3 \|k\|^2_{L_6(\Om)} + c(1/\eps_3)\|\nb
v\|^2_{L_2(\Om)} \|k\|^2_{L_3(\Om)} \eenn or by \benn
\int_{\Om}k^2|\nb v| dx \le \eps_3 \|k\|^2_{L_6(\Om)} +
c(1/\eps_3)\|\nb v\|^2_{L_3(\Om)} \|k\|^2_{L_2(\Om)}. \eenn To
examine the third term on the r.h.s. of (\ref{k=}) we express it
in the form \benn \Oi b \nb k k dx + \Oi \nb \va \nb k k dx \equiv
I_1 +I_2 \eenn In view of \cite{RZ1} we have \benn |I_1| \le
c\rho^{1/6} \|\d\|_{H^1(\Om)} \|k\|^2_{H^1(\Om)} \eenn and \benn
|I_2| \le c (\eps_4 \rho^{\mu-2/3} \sup_{x_3} \|\d_3\|_{L_3(S_2)}
+\rho^{\mu+1/3}\sup_{x_3} \|\d_{,x_3}\|_{L_3(S_2)})
\|k\|^2_{H^1(\Om)} \eenn The fourth term we estimate by \benn
\eps_5 \|k\|_{L_6(\Om)}^2 + c(1/\eps_5) \|v\|_{L_6(\Om)}^2 \|\nb
\h\|^2_{L_{3/2}(\Om)} \eenn The fifth term we treat as follows
\benn \eps_6 \|k\|_{L_6(\Om)}^2 + c(1/\eps_6) \|\h\|_{L_3(\Om)}^2
\|\nb v\|^2_{L_{2}(\Om)},\eenn the sixth by \benn \eps_7
\|k\|_{L_6(\Om)}^2 + c(1/\eps_7) \|\h_t\|^2_{L_{6/5}(\Om)}
\eenn and finally the last one by \benn \eps_8 \|k\|_{L_6(\Om)}^2
+ c(1/\eps_8) \|g\|^2_{L_{6/5}(\Om)} \eenn Using the above
considerations in (\ref{k=}) and assuming that $\eps_1-\eps_8$ are
sufficiently small we obtain \ben \label{est-1} \bal \frac{d}{dt}
\Oi k^2 dx + \nu \|k\|_{H^1(\Om)}^2+ \ga
\sum_{\al=1}^2\|k\cdot\bar{\tau}_{\al}\|_{L_2(S_1)}^2 \\ \le
 c\bigg[\|\De'd\|^2_{L_{2}(S_2)} +\|d_t\|^2_{L_{4/3}(S_2)}
 +\|f_3\|^2_{L_{4/3}(S_2)}+ \|\h\|^2_{H^1(\Om)} \\ + \|d\|^6_{L_3(S_2)}
(\|h_3\|^2_{L_2(\Om)}
 +\|\h_3\|^2_{L_2(\Om)}) +\|\nb v\|^2_{L_2(\Om)}
(\|h_3\|^2_{L_3(\Om)}  + \|\h_3\|^2_{L_3(\Om)}) \\ +
\|v\|^2_{L_6(\Om)} \|\nb \h\|^2_{L_{3/2}(\Om)}+\|\h\|^2_{L_3(\Om)}
\|\nb v\|^2_{L_2(\Om)}+  \|h_t\|^2_{L_{6/5}(\Om)} +
\|g\|^2_{L_{6/5}(\Om)}\biggr] . \eal \een Integrating
(\ref{est-1}) with respect to time and using the energy type bound
of the form \benn \|v\|_{V^0_2(\Om^t)} \le {A},\ \  t\le T, \eenn
where $A$ is defined in Lemma~\ref{def-A} we get \ben \bal
\label{k-h} \|k\|^2_{V_2^0(\Om^t)} + \ga \sum_{\al=1}^2\|k\cdot
\bar{\tau}_{\al}\|^2_{L_2(S_1^t)} \le c
 \bigg[\|\De'd\|^2_{L_2(0,t;L_{2}(S_2))} +\|d_t\|^2_{L_2(0,t;L_{4/3}(S_2))} + \|f_3\|^2_{L_2(0,t;L_{4/3}(S_2))}
\\ +\|\h\|^2_{H^1(\Om^t)}
 +\|d\|^6_{L_{\infty}(0,t;L_{3}(S_2))}(
\|h_3\|^2_{L_2(\Om^t)} + \|\h_3\|^2_{L_2(\Om^t)})\bigg] \\+
\sup_{\tau} A^2 \|h\|^2_{L_3(\Om)}+ \sup_{\tau} A^2 \|\nb
\h\|^2_{L_{3/2}(\Om)} + \sup_{\tau} A^2 \|\h\|^2_{L_3(\Om)} +
\|\h_t\|^2_{L_2(\Om^t)} \\ + \|g\|^2_{L_2(0,t;L_{6/5}(\Om))}+
\|h(0)\|^2_{L_2(\Om)}+ \|\h(0)\|^2_{L_2(\Om)}. \eal \een In view
of Lemma~\ref{h-fala} we have \ben \bal \label{k-hA}
\|k\|^2_{V_2^0(\Om^t)} + \ga \sum_{\al=1}^2\|k\cdot
\bar{\tau}_{\al}\|^2_{L_2(S_1^t)} \le c
\bigg[\|d_{,x'}\|^2_{L_2(0,t;H^1(S_2))}
+\|d_t\|^2_{L_2(0,t;L_{4/3}(S_2))}
\\ \quad \quad \quad + \|f_3\|^2_{L_2(0,t;L_{4/3}(S_2))}
+\|d\|^6_{L_{\infty}(0,t;L_{3}(S_2))}\|h_3\|^2_{L_2(\Om^t)} \bigg]
\\+ A^2( \sup_{\tau} \|h\|^2_{L_3(\Om)}+
A^2\|d_{,x'}\|^2_{L_{\infty}(0,t;H^1(S_2))}) + A^2 \sup_{\tau}
\|d_{,x't}\|^2_{L_2(S_2)} \\+ \|g\|^2_{L_2(0,t;L_{6/5}(\Om))}+
\|h(0)\|^2_{L_2(\Om)}. \eal \een We set \benn D_1 =
\|d_t\|^2_{L_2(0,t;H^1(S_2))} +
\|d_{,x'}\|^2_{L_2(0,t;H^1(S_2))} \\
D_2 = \|d_{,x'}\|^2_{L_{\infty}(0,t;H^1(S_2))} \eenn Then
(\ref{k-hA}) takes the form \ben \bal \label{k-D}
\|k\|^2_{V_2^0(\Om^t)} + \ga \sum_{\al=1}^2\|k\cdot
\bar{\tau}_{\al}\|^2_{L_2(S_1^t)} \le c\biggl[ D_1^2
+\|f_3\|^2_{L_2(0,t;L_{4/3}(S_2))} \\
+(\|d\|^6_{L_{\infty}(0,t;L_{3}(S_2))}+A^2+1)\left(D_2^2
+\|h\|^2_{L_2(\Om^t)}+ \|h\|^2_{L_{\infty}(0,t;L_3(\Om))} \right)
\\ + \|g\|^2_{L_2(0,t;L_{6/5}(\Om))}+ \|h(0)\|^2_{L_2(\Om)}\biggr].
\eal \een Therefore, defining quantities $\La, D_1, D_2, H_1$ as
in (\ref{D-def}) and (\ref{H-def}) we can formulate the following
compact version of the last inequality. \benn
\|k\|^2_{V_2^0(\Om^t)} +
\sum_{\al=1}^2\|k\cdot \bar{\tau}_{\al}\|^2_{L_2(S_1^t)} \\
\le c (\|d\|^6_{L_{\infty}(0,t;L_{3}(S_2))}+A^2+1)(\La^2 + H_1^2)
\eenn and this concludes the proof. \end{proof}

Consequently, we can show the following inequality for $h$
\begin{lemma}
Let $D_0, \La$ are finite. Then a solutions $h$ of (\ref{h-q})
satisfies \ben \bal \label{h-v0} \|h\|^2_{V_2^0(\Om^t)} +
\sum_{\al=1}^2\|h\cdot \bar{\tau}_{\al}\|^2_{L_2(S_1^t)} \le c
D_0(\La^2 + H_1^2) \eal \een
 where \benn
 D_0 =\|d\|^6_{L_{\infty}(0,t;L_{3}(S_2))} +A^2 +1 \eenn
\end{lemma}
\begin{proof}
Using (\ref{k-Dcomp}), $h=k+\h,$ and the estimate \benn
\|\h\|^2_{V_2^0(\Om^t)} + \sum_{\al=1}^2\|\h\cdot
\bar{\tau}_{\al}\|^2_{L_2(S_1^t)} \le c  (D_1^2 + D_2^2) \eenn we
conclude the result. \end{proof}
\begin{remark}
Since we want to admit arbitrary large flux, the inequality above
does not involve any small parameter that we could use in the
existence proof. Thus, we will need some more refined relation and
this can be achieved by improving the regularity through the
system for vorticity component $\chi$.

\end{remark}

%\chi function

\subsection{A priori estimates for vorticity component $\chi$}
We need to examine solutions of problem (\ref{v-chi}) and
consequently, (\ref{chi}). We have to deal with possibly non-zero
boundary condition of function $\chi$ on $S_1^T$ which we set as
$\chi_*$. To this end, let us introduce a function $\ci$ as a
solution to the following problem \ben \label{ci} \bal
\ci_{,t} - \nu\De\ci & = 0 \quad {\rm in} \ \ \Om^T, \\
\ci & = \chi_* \quad {\rm on} \ \ S_1^T, \\
\ci_{,x_3} & = 0,  \quad {\rm on} \ \ S_2^T, \\
\ci|_{t=0} & = \chi_0 \quad {\rm in} \ \ \Om.
\eal \een

To show the existence of such function we need the following
compatibility conditions \ben \bal \chi_{*,x_3} & =
 0, \ {\rm on} \ \bar{S}_1\cap \bar{S}_2\\
\chi_0|_{S_1} & = \chi_*|_{t=0}, \\
\chi_{0,x_3} & = 0 \ {\rm on} \  S_2. \eal \een where \benn
\chi_{*,x_3} = - \sum_{i,j=1}^2 \left[d_{,x_i} ( n_{i,x_j}
\tau_{1j} + \tau_{1i,x_j} n_j)+ \frac{\ga}{\nu} d_{,x_j}\tau_{1j}+
d_{x_i}\tau_{1i}(\tau_{12,x_1}- \tau_{11,x_2})\right] \eenn%
We note that $\chi_0$ depends of $v$ in a similar way as $\chi_*.$

Then, we consider the new function $\chi' = \chi - \ci$ which is a
solution to the following problem \ben \label{chi'} \bal
&\chi'_{,t}+v\cdot\nabla\chi'-h_3\chi'+h_2v_{3,x_1}-h_1v_{3,x_2}
 -\nu\Delta\chi'\\ & =F_3- v\cdot\nabla\ci+h_3\ci \quad &{\rm in}\ \ \Omega^T,\\
&\chi' = 0\quad &{\rm on}\ \ S_1^T,\\  &\chi'_{,x_3}=0\quad &{\rm
on}\ \ S_2^T,\\ &\chi'\big|_{t=0}=\chi(0)\quad &{\rm in}\ \
\Omega. \eal \een%

\begin{lemma}
Assume that $h\in L_{\infty}(0,t;L_3(\Om)), v'\in L_{\infty}(0,t;
H^1(\Om))\cap L_2(0,t;H^2(\Om)) \cap L_2(\Om;H^{1/2}(0,t)),
\chi(0)\in L_2(\Om), F_3 \in L_2(0,t;L_{6/5}(\Om)).$ Then
solutions to problem (\ref{chi}) satisfy the inequality \ben \bal
\label{chi-v} \|\chi\|_{V_2^0(\Om^t)} \le \va(c_2, A)\sup_t
\|h(t)\|_{L_3(\Om)} + \frac{1}{\eps} \va(c_2, c_3, A) \sup_t
\|h(t)\|^2_{L_3(\Om)} \\ + \frac{1}{\eps} \va(c_2, c_3, A)
+\frac{\ga c_2}{\nu} \|F_3\|_{L_2(0,t; H^2(\Om))} + \eps
(\|v'\|_{L_{\infty}(0,t;H^1(\Om))}\\ +
\|v'\|_{L_2(0,t;H^2(\Om))})+ c_3
\|v'\|_{L_2(\Om;H^{1/2}(0,t))} + \|\chi(0)\|_{L_2(\Om)} \\
\equiv c(H_1 + H_1^2)  + c \|v'\|_{L_2(\Om;H^{1/2}(0,t))} +
\eps(\|v'\|_{L_{\infty}(0,t;H^1(\Om))} +
\|v'\|_{L_2(0,t;H^2(\Om))}) + c \La_2, \eal \een where $\eps \in
(0,1), a>0, t\in (0,T), c_3$ is some constant, \benn \La_2 =
\|F_3\|_{L_2(0,t; H^2(\Om))}+ \|\chi(0)\|_{L_2(\Om)}+ 1 \eenn and
$\va$ is a generic function which changes its form from formula to
formula.
\end{lemma}
\begin{proof}
 We use the first equation of the system (\ref{chi'}): we multiply it by $\chi',$ integrate over $\Om$ and use boundary conditions that yields
 \ben \label{chi-prim} \bal
 \frac{1}{2} \frac{d}{dt} \|\chi'\|^2_{L_2(\Om)} + \nu \|\nb \chi'\|^2_{L_2(\Om)}= \Oi h_3 \chi'^2 dx - \Oi (h_2v_{3,x_1}-h_1v_{3,x_2}) \chi' dx +\Oi
 F_3\chi' dx \\ - \Oi v\cdot\nabla\ci\chi' dx +\Oi h_3\ci \chi' dx
 \eal \een
The first term on the r.h.s. of (\ref{chi-prim}) can be bounded by
\benn |\Oi h_3 \chi'^2 dx| = |\Oi h_3\chi'(\chi - \ci) dx| \le
|\Oi h_3 \chi'(\chi -\ci)dx| \le |\Oi h_3 \chi'\chi dx|+ |\Oi h_3
\chi'\ci dx| \\ \le \frac{\eps_1}{4}\|\chi'\|^2_{L_6(\Om)} +
\frac{1}{\eps_1}\|h_3\|^2_{L_3(\Om)}\|\chi\|^2_{L_2(\Om)}
+\frac{\eps_1}{4}\|\chi'\|^2_{L_6(\Om)} +
\frac{1}{\eps_1}\|h_3\|^2_{L_2(\Om)}\|\ci\|^2_{L_3(\Om)} \eenn the
second as follows \benn \frac{\eps_2}{2} \|\chi'\|^2_{L_6(\Om)} +
\frac{1}{2\eps_2} \|h\|^2_{L_3(\Om)} \|v_{3,x'}\|^2_{L_2(\Om)},
\eenn and the third one: \benn \frac{\eps_3}{2}
\|\chi'\|^2_{L_6(\Om)} + \frac{1}{2\eps_3}
\|F_3\|^2_{L_{6/5}(\Om)}. \eenn The fourth term on the r.h.s. of
(\ref{chi-prim}) we express in the form \benn \Oi v\cdot\nb
\chi'\ci dx \eenn and estimate as follows \benn \frac{\eps_4}{2}
\|\nb \chi'\|^2_{L_2(\Om)} + \frac{1}{2\eps_4} \|v\|^2_{L_6(\Om)}
\|\chi'\|^2_{L_3(\Om)}. \eenn Finally, the last term on the r.h.s.
of (\ref{chi-prim}) is bounded by \benn \frac{\eps_5}{2}
\|\chi'\|^2_{L_6(\Om)} + \frac{1}{2\eps_5} \|\chi'\|^2_{L_3(\Om)}.
\eenn Using the above estimates in (\ref{chi-prim}), setting
$\eps_1=\eps_2=\eps_3=\eps,$ $\eps_4=\frac{\nu}{2},$
$\eps=\frac{\nu}{8c_2}$ we obtain \benn
\frac{d}{dt}\|\chi'\|^2_{L_2(\Om)} + \nu \|\nb
\chi'\|^2_{L_2(\Om)} \le \frac{8c_2}{\nu} (\|\chi\|^2_{L_2(\Om)} +
\|v_{3,x'}\|^2_{L_2(\Om)}) \|h\|^2_{L_3(\Om)} \\ +
\left(\frac{2}{\nu} \|v\|^2_{L_6(\Om)} +
\frac{16c_2}{\nu}\|h\|^2_{L_2(\Om)}\right) \|\ci\|^2_{L_3(\Om)} +
\frac{8c_2}{\nu} \|F_3\|^2_{L_{6/5}(\Om)}. \eenn Integrating the
above inequality with respect to time and using estimate for the
weak solutions we can find the following \benn
\|\chi'\|^2_{L_2(\Om)}+ \nu \|\nb \chi'\|^2_{L_2(\Om^T)} \le
\frac{16c_2A^2}{\nu} \sup_t \|h(t)\|^2_{L_3(\Om)} + \frac{16c_2
+2}{\nu} A^2 \sup_t \|\ci\|^2_{L_3(\Om)}\\ + \frac{8c_2}{\nu}
\|F_3\|^2_{L_2(0,T;L_{6/5}(\Om))} + \|\chi(0)\|^2_{L_2(\Om)}.\eenn
Since $\chi = \chi'+\ci$ we obtain \ben \bal \label{chi-v2}
\|\chi\|^2_{V_2^0(\Om^T)}
 \le \frac{16c_2A^2}{\nu} \sup_t \|h(t)\|^2_{L_3(\Om)} + \frac{16c_2
+2}{\nu} A^2 \sup_t \|\ci\|^2_{L_3(\Om)} \\ + \frac{8c_2}{\nu}
\|F_3\|^2_{L_2(0,T;L_{6/5}(\Om))} + \|\chi(0)\|^2_{L_2(\Om)}.\eal
\een In view of Lemma and some interpolation inequalities we have
\ben \label{inter} \bal \|\ci\|_{L_2(0,T;L_2(\Om))} \le c'_1
\|v'\|_{L_2(0,T;L_2(S_1))}
 \le \eps \|v'\|_{L_2(0,T;H^1(\Om))} +c_2'\eps^{-1} A, \\
 \|\ci\|_{L_{\infty}(0,T;L_3(\Om))} \le c'_3
 \|v'\|_{L_{\infty}(0,T;L_3(S_1))}\le \eps^{1/6}
 \|v'\|_{L_{\infty}(0,T;H^1(\Om))}+c'_4 \eps^{-5/6} A, \\
 \|\ci\|_{L_{\infty}(0,T;L_2(\Om))} \le c'_5
 \|v'\|_{L_{\infty}(0,T;L_2(S_1))} \\ \le \eps^{1/2}
 \|v'\|_{L_{\infty}(0,T;H^1(\Om))}+c'_6 \eps^{-1/2} A, \\
\|\nb \ci\|_{L_2(0,T;L_2(\Om))} \le c'_7
\|v'\|_{W^{1/2,1/4}_2(S_1^T)} \le c'_8 \|v'\|_{W^{1,1/2}_2(\Om^T)}
\\ = c'_8 (\|v'\|_{L_2(0,T;H^1(\Om))} + \|v'\|_{L_2(\Om;H^{1/2}(0,T))})
\\ \le \eps^{1/2} \|v'\|_{L_2(0,T;H^2(\Om))} +\eps^{-1/2} c'_9A +
c'_8 \|v'\|_{L_2(\Om;H^{1/2}(0,T))}
 \eal \een where $c'_1, \ldots,c'_9$ are constants from corresponding
theorems and of imbedding. Let $c'_i \le c'_3, i=1,\ldots,9.$ Then
we use (\ref{chi-v2}) and (\ref{inter}) to obtain (\ref{chi-v}).
This concludes the proof.
\end{proof}

\subsection{(rot, div) problem and inequalities relating $v$ and $h$}

Let us consider the elliptic (rot, div) problem (\ref{v-chi}),
i.e. \benn
v_{1,x_2}-v_{2,x_1} = \chi\quad &{\rm in}\ \ \Omega',\\
v_{1,x_1}+v_{2,x_2} =  -h_3 \quad &{\rm in}\ \ \Omega',\\
v'\cdot\bar n' =0\quad &{\rm on}\ \ S'_1,  \eenn For solutions
$v'$ of this system, in view of inequalities for $h$ and $\chi$,
i.e. (\ref{h-v0}) and (\ref{chi-v}), we derive at \ben \bal
 \label{v'} \|v'\|^2_{V^1_2(\Om^T)} & \leq  c D_0 (\La + H_1) \\  & + c
\|v'\|_{L_2(\Om;H^{1/2}(0,t))} + c(H_1 + H_1^2) + c \La_2
 \eal \een %

Let us consider problem (\ref{NS}) written in the form of
Stokes-type system \ben \label{ns-new}
\bal &v_{,t}-\Div\T(v,p)=-v'\cdot\nabla v-wh +f, \quad &{\rm in}\ \Om^T \\
&\Div v=0, \quad & {\rm in}\ \Om^T \\ &v\cdot\bar n =0, \ \bar
n\cdot\D(v)\cdot\bar\tau_\al + \ga v \cdot\bar\tau_\al =0,\ \
 \alpha=1,2, \quad & {\rm on}\ S_1^T \\
&v\cdot\bar n =d, \ \bar n\cdot\D(v)\cdot\bar\tau_\al=0,\ \
 \alpha=1,2, \quad & {\rm on}\ S_1^T \\
&v|_{t=0}=v(0), \quad & {\rm in}\ \Om. \eal \een This yields the
next result
\begin{lemma} \label{3.7}
Let the assumptions of Lemmas~3.1-3.5 be satisfied. Let
\mbox{$h\in L_{\frac{10}{3}}(\Om^T),$} $f\in L_2(\Om^T), v(0)\in
H^1(\Om).$ Then for solutions of (\ref{ns-new}) we obtain the
inequality \ben \label{v-p} \bal \|v\|_{W^{2,1}_2(\Om^t)} + \|\nb
p\|_{L_2(\Om^t)} \leq cD_0^2 (H_1^2 + \La^2) + c(H_1 + H_1^2) \\ +
c\La_2^2 + c \La_3^2 + c\La_4 , \quad t\leq T, \eal \een where
$\La_3, \La_4$ are defined by (\ref{la3-4}).
\end{lemma}

\begin{proof}
In view of \cite{Z2}, Lemma~3.7  we have that \benn
\|v'\|_{L_{10}(\Om^T)} \leq c\|v'\|_{V^1_2(\Om^T)}. \eenn Hence
\benn \|v'\nb v\|_{L_{5/3}(\Om^T)}& \leq &  c A
\|v'\|_{V^1_2(\Om^T)} , \\
\|wh\|_{L_{5/3}(\Om^T)} & \leq & cA \|h\|_{L_{10/3}(\Om^T)} .
\eenn In view of
the above estimates we have %
\ben \label{v-est}\bal \|v\|_{W^{2,1}_{5/3}(\Om^T)} \leq cA
(\|v'\|_{V^1_2(\Om^T)} + \|h\|_{L_{10/3}(\Om^T)})\\ +
c(\|f\|_{L_{5/3}(\Om^T)} + \|v(0)\|_{W_{5/3}^{4/5}(\Om)}). \eal
\een We calculate \benn \|v'\nb v\|_{L_2(\Om^T)} & \leq &
\|v'\|_{L_{10}(\Om^T)} \|\nb v\|_{L_{5/2}(\Om^T)} \\ & \leq &
\|v'\|_{V_{2}^1(\Om^T)} \|v\|_{W^{2,1}_{5/3}(\Om^T)}, \\
\|wh\|_{L_2(\Om^T)} & \leq &
\|w\|_{L_5(\Om^T)}\|h\|_{L_{10/3}(\Om^T)} \\ & \leq &
c\|v\|_{W^{2,1}_{5/3}(\Om^T)}\|h\|_{L_{10/3}(\Om^T)}. \eenn
Applying (\ref{v'}) in (\ref{v-est}) and using the interpolation
\benn \|v'\|_{L_2(0,T;H^{1/2}(\Om))} \leq \eps
\|v'\|_{W^{2,1}_{5/3}(\Om^T)} + c(1/\eps)A, \eenn we obtain \ben
\label{v-k1} \bal \|v\|_{W^{2,1}_{5/3}(\Om^T)} \leq cD_0(H_1 +
\La) + c(H_1 + H_1^2) + c \La_2 + c\La_3, \eal \een and \ben
\label{v-p1} \bal \|v\|_{W^{2,1}_2(\Om^T)} \leq cD_0^2 (H_1^2 +
\La^2) + c(H_1 + H_1^2) + c\La_2^2 + c \La_3^2 + c\La_4 \eal \een
 where %
\benn   \La_3 =  \|f\|_{L_{5/3}(\Om^T)} +
\|v(0)\|_{W_{5/3}^{4/5}(\Om)}, \\ \La_4 = \|f\|_{L_2(\Om^T)} +
\|v(0)\|_{H^1(\Om)} .  \eenn%
Applying the above estimates we conclude for solutions to problem
(\ref{ns-new}) the inequality (\ref{v-p}). This yields the thesis
of the lemma.
\end{proof}
Next, we need to find bound for $h$ in terms of $v$ with some
small parameter. We prove
\begin{lemma} \label{h-ex} Let the assumptions of Lemmas~3.1-3.5 be satisfied. Let
$\nb v \in L_3(\Om)$ and $D_0, \La$ are finite. Then a solution
$h$ of (\ref{h-q}) satisfies \ben \label{h-exp}
\|h\|^2_{V^0_2(\Om^t)} \le D_0 \exp \left[\int_0^t
(\|d\|^6_{L_3(S_2)} + \|\nb v\|^2_{L_3(\Om)})dt\right] \La^2. \een
\end{lemma}
\begin{proof}
%For the problem (\ref{h-q}) we can derive \benn
%\|h\|_{W^{2,1}_2(\Om^T)} \leq \va(\|v\|_{W^{2,1}_2(\Om^T)})
%\|h\|_{L_2(\Om^T)} + \|g\|_{L_2(\Om^T)} % + \|h(0)\|_{L_2(\Om)}
%\eenn Since $\va(\|v\|_{W^{2,1}_2(\Om^T)})$ can be expressed via
%$h$ by Lemma~\ref{3.7} and inequality (\ref{v-p}), it remains to
%estimate $\|h\|_{L_2(\Om^T)}$ which is the part of $H_1$.
To this
end we consider again the inequality \benn
 \frac{d}{dt} \Oi k^2 dx + {\nu} \|k\|_{H^1(\Om)}^2 +\ga
\|k\cdot\bar{\tau}_{\al}\|_{L_2(S_1)}^2
\\ \le c(\|d_t\|^2_{L_{4/3}(S_2)} + \|\De'
d\|^2_{L_{4/3}(S_2)}+\|f_3\|^2_{L_{4/3}(S_2)} +
\|\h\|^2_{H^1(\Om)} )
\\+ \|d\|^6_{L_3(S_2)} \|k\|^2_{L_2(\Om)} +
|\Oi G h dx |
\eenn %
 and to examine the last term on the r.h.s. we use the (\ref{k=})
 and the relation
 \benn \Oi |\nb v| k^2 dx \le \eps_3 \|k\|^2_{L_6(\Om)} +
c(1/\eps_3)\|\nb v\|^2_{L_3(\Om)} \|k\|^2_{L_2(\Om)}. \eenn We
deal with inequalities above and through such modifications, we
have \benn\|k\|^2_{V_2^0(\Om^t)} + \ga \sum_{\al=1}^2\|k\cdot
\bar{\tau}_{\al}\|^2_{L_2(S_1^t)} \le c \bigg[\|\nb
v\|^2_{L_2(0,t;L_3(\Om))}+\|d\|^6_{L_{\infty}(0,t;L_{3}(S_2))}\bigg]
\|k\|^2_{L_2(\Om^t)}
\\ + \|d_{,x'}\|^2_{L_2(0,t;H^1(S_2))}
+\|d_t\|^2_{L_2(0,t;L_{4/3}(S_2))} +
\|f_3\|^2_{L_2(0,t;L_{4/3}(S_2))} \\ +
A^2\|d_{,x'}\|^2_{L_{\infty}(0,t;H^1(S_2))}) + A^2 \sup_{\tau}
\|d_{,x't}\|^2_{L_2(S_2)} + \|g\|^2_{L_2(0,t;L_{6/5}(\Om))}+
\|h(0)\|^2_{L_2(\Om)}. \eenn Then, instead of (\ref{h-v0}) we can
obtain \benn
 \|k\|^2_{V^0_2(\Om^t)} \le D_0
\exp \left[\int_0^t (\|d\|^6_{L_3(S_2)} + \|\nb
v\|^2_{L_3(\Om)})dt\right] \La^2 \eenn and using
Lemma~\ref{h-fala} yields (\ref{h-exp}).
\end{proof}

\begin{lemma} \label{A-lemma} Let the Conditions 1 and 2 hold. Then
for a weak solution $v$ to (\ref{NS}) and $h=v,_{x_3}$ - a
solution to (\ref{ns-new}) there exists such constant $\mathcal A$
that \ben \label{h-A} \bal {\bf H}(T)
& \equiv \|h\|_{W_2^{2,1}(\Om^T)} \le \mathcal A \\
{\rm and}  &  \\  \|v\|_{W^{2,1}_2(\Om^T)} & \leq cD_0^2
(\mathcal{A}^2 + \La^2) + c(\mathcal{A} + \mathcal{A}^2)  +
c\La_2^2 + c \La_3^2 + c\La_4 \eal \een
\end{lemma}
\begin{proof}
We note that \benn \|\nb v\|_{L_2(0,T;L_3(\Om))} \le c
\|v\|_{W_{5/3}^{2,1}(\Om^{T})} \le  cD_0(H_1 + \La) + c(H_1 +
H_1^2) + c \La_2 + c\La_3 \eenn %
where we used (\ref{v-k1}). Applying also the inequality
(\ref{h-exp}) from Lemma~\ref{h-ex} we have \benn
\|h\|^2_{L_2(\Om^T)} \le D_0 \exp \left[\int_0^t
(\|d\|^6_{L_3(S_2)}dt + cD_0(H_1 + \La) + c(H_1 + H_1^2) + c \La_2
+ c\La_3 \right]\La^2 \eenn Since \benn H_1(T) \le c {\bf H}(T)
\eenn we obtain, by formulas above \benn \bal {\bf H}^2(T) \le c
D_0^2 \La^2\left[({\bf H}(T) + \La) + c({\bf H}(T) + {\bf H}(T)^2)
+ c \La_2 + c\La_3 \right] \\ \cdot \exp \left[\int_0^t
(\|d\|^6_{L_3(S_2)}dt + cD_0({\bf H}(T) + \La) + c({\bf H}(T) +
{\bf H}(T)^2) + c \La_2 + c\La_3 \right] \\ + \|g\|^2_{L_2(\Om^T)}
+ \|h(0)\|^2_{H^1(\Om)} \eal \eenn To obtain the result, we need
\ben \bal  c D_0^2 \La^2 \left[({\mathcal A} + \La) +
c({\mathcal A} + {\mathcal A}^2) + c \La_2 + c\La_3 \right] \\
\exp \left[\int_0^t (\|d\|^6_{L_3(S_2)}dt + cD_0({\mathcal A} +
\La) + c({\mathcal A} + {\mathcal A}^2) + c \La_2 + c\La_3 \right]
\\ + \|g\|^2_{L_2(\Om^T)} + \|h(0)\|^2_{H^1(\Om)}\le {\mathcal A}^2 \eal
\een which can be satisfied for sufficiently small
$\|g\|_{L_2(\Om^T)},\|h(0)\|_{H^1(\Om)}$ and
$$\La = D_1+ D_2
+\|f_3\|^2_{L_2(0,t;L_{4/3}(S_2))} +
\|g\|^2_{L_2(0,t;L_{6/5}(\Om))}+ \|h(0)\|^2_{L_2(\Om)},$$ where \benn D_1 & = & \|d_t\|^2_{L_2(0,t;H^1(S_2))} + \|d_{,x'}\|^2_{L_2(0,t;H^1(S_2))}, \\
D_2 & = &  \|d_{,x'}\|^2_{L_{\infty}(0,t;H^1(S_2))}. \eenn

To obtain the bound for $v$ in terms of $\mathcal{A}$ we use
(\ref{v-p1}) and $(\ref{h-A})_1$. This concludes the proof.
\end{proof}

%**********************************************************************

%Since \benn H_1(T) \le c {\bf H}(T) \eenn and by (\ref{h-v0})
%\benn {\bf H}(T) \equiv \|h\|_{W^{2,1}_2(\Om^T)} \leq c
%D_0^2(\La^2 + H_1(T)^2) \eenn so it is enough to satisfy the
%inequality \benn c D_0^2(\La^2 + {\mathcal A}^2) \le {\mathcal A}
%\eenn which can be fulfilled for sufficiently small
%$D_0=\|d_1\|^2_{L_{\infty}(0,t;L_{3}(S_2))} +A +1. $

%************************************************************************

\section{Existence}
\setcounter{equation}{0} %
\setcounter{lemma}{0}

\begin{lemma}
Assume that Conditions~1 and 2 are satisfied. Then the solution
$(v,p)$ to (\ref{NS}) satisfying (\ref{main}) exists on $(0,T)$
for some $T>0.$
\end{lemma}
\begin{proof}
To prove the existence of solutions to problem (\ref{NS}) we will
use the Leray-Schauder theorem. To this end, we construct the
mappings

\ben \label{NS-lambda} \bal %
&v_{t}-\Div \T(v,p)=- \la \v\cdot\nb \v+ f\quad &{\rm in}\
 \Om^T=\Om\times(0,T),\\
&\Div v=0\quad &{\rm in}\  \Om^T,\\
 &v\cdot\bar n=0\quad &{\rm
on}\  S_1^T ,\\
&\nu \bar n\cdot\D(v)\cdot\bar\tau_\alpha + \ga v \cdot
\bar\tau_{\al}=0,\  \alpha=1,2,\quad
 &{\rm on}\  S_1^T,\\
 &v\cdot\bar n=d\quad &{\rm
on}\  S_2^T ,\\
 & \bar n\cdot\D(v)\cdot\bar\tau_\alpha =0,\ \
\alpha=1,2,\quad
 &{\rm on}\ S_2^T,\\
&v\big|_{t=0}=v(0)\quad &{\rm in}\  \Om, \eal \een

and

\ben \label{h-q-lambda} \bal h_{,t}-\Div \T(h,q)& = -\la
(\v\cdot\nb \h+\h\cdot
 \nb \v) +g\quad &{\rm in} \ \Om^T,\\
\Div h &= 0\quad &{\rm in} \ \Om^T,\\ n\cdot\bar h &= 0,\ \ \bar
n\cdot\D(h)\cdot\bar\tau_\alpha + \ga h \cdot\bar\tau_\alpha =0,\
\  \alpha=1,2\quad &{\rm on} \ S_1^T,\\
h_i &= -d_{x_i},\ \ i=1,2,\ \ h_{3,x_3}= \De' d \quad &{\rm on} \ S_2^T,\\
h\big|_{t=0}&= h(0)\quad &{\rm in} \ \Om.\eal \een where $g=
f_{,x_3}, \De' = \pa_{x_1}^2 + \pa_{x_2}^2,$
 $\lambda\in[0,1]$ and $\v$, $\h$ are treated as
given functions. We assume that $\h=\v_{,x_3}$, thus
differentiating (\ref{NS-lambda}) with respect to $x_3$ and
subtracting from (\ref{h-q-lambda}) we obtain that
$$
h=v_{,x_3}.
$$
Problems (\ref{NS-lambda}), (\ref{h-q-lambda}) define the
mappings%
 \benn &\Phi_1:\ (\v,\la)\to (v,p),\\ &\Phi_2:\
(\v,\h,\la)\to(h,q). \eenn We set $\Phi=(\Phi_1,\Phi_2)$. In
previous Section we have shown a-priori estimate for a~fixed point
of $\Phi$ for $\la=1$. On the other hand, for $\lambda=0$ we have
a~unique existence of solutions to problems (\ref{NS-lambda}) and
(\ref{h-q-lambda}).

Let us introduce the space \benn {\mathcal
M}(\Om^{T})=L_{2r}(0,(T;W_{6\eta\over3+\eta}^2(\Om)), \quad
\eta\ge2,\ \ r\ge2. \eenn We shall find restrictions on $r,\eta$
such that \benn \Phi:\ {\mathcal M}(\Om^{(T})\times{\mathcal
M}(\Om^{T})\to {\mathcal M}(\Om^{T})\times{\mathcal M}(\Om^{T})
\eenn is a~compact mapping.

\noindent Assume that $\v\in
L_{2r}(0,T;W_{\frac{6\eta}{3+\eta}}^1(\Om))$. Then \ben \bal
\label{v-nbv} &\|\v\cdot\nb\v\|_{L_r(0,T;L_\eta(\Om))}=
\bigg(\int_{0}^{T}dt
\|\v\cdot\nb\v\|_{L_\eta(\Om)}^r\bigg)^{1/r}\cr
&\le\bigg(\int_{0}^{T}dt\|\v\|_{L_{\frac{6\eta}{3-\eta}}(\Om)}^r
\|\nb\v\|_{L_{\frac{6\eta}{3+\eta}}(\Om)}^r\bigg)^{1/r}\\
&\le c\bigg(\intop_{0}^{T}dt
\|\v\|_{W_{\frac{6\eta}{3+\eta}}^1(\Om)}^{2r}\bigg)^{1/r}\le c
\|\v\|_{L_{2r}(0,T;W_{\frac{6\eta}{3+\eta}}^1(\Om))}^2 \eal \een
In the same way we obtain \ben \bal \label{v-nbh}
&\|\v\cdot\nb\h\|_{L_r(0,T;L_\eta(\Om))}+
\|\h\cdot\nb\v\|_{L_r(0,T;L_\eta(\Om))}\\
&\le c\|\v\|_{L_{2r}(0,T;W_{6\eta\over3+\eta}^1(\Om))}
\|\h\|_{L_{2r}(0,T;W_{\frac{6\eta}{3+\eta}}^1(\Om))}. \eal \een In
view of (\ref{v-nbv}) and (\ref{v-nbh}) can be shown that
solutions to problems (\ref{NS-lambda}) and (\ref{h-q-lambda})
belong to $W_{\eta,r}^{2,1}(\Om^{T})$ (see \cite{S1}).

\noindent We are going to use the following imbeddings \ben
\label{imb1} W_2^{2,1}(\Om^{T})\supset W_{\eta,r}^{2,1}(\Om^{T})
\een and \ben \label{imb2} W_2^{2,1}(\Om^{T})\subset
L_{2r}(0,T;W_{\frac{6\eta}{3+\eta}}^1(\Om)) \equiv{\mathcal
M}(\Om^{T}),  \een where (\ref{imb1}) holds for $\eta\ge2$,
$r\ge2$ and (\ref{imb2}) is compact for $r,\eta$ satisfying the
inequality \benn
\frac{5}{2}-\frac{3}{\frac{6\eta}{3+\eta}}-\frac{2}{2r}<1 \eenn
which takes the form \ben \label{imb3} 1<
\frac{3}{2\eta}+\frac{1}{r}.  \een Setting $r=\eta=2$ we obtain
that $\v,\h\in L_4(0,T;W_{12/5}^1(\Om))$ and then condition
(\ref{imb3}) takes the form \ben \label{imb}  1< \frac{3}{4} +
\frac{1}{2} \quad {\rm so}\quad \frac{1}{2} < \frac{3}{4}.  \een
Hence, we have compactness of mappings $\Phi_1$ and $\Phi_2$.

\noindent To show continuity of mappings $\Phi_1$ and $\Phi_2$ we
consider \benn &
v_{st}-\Div\T(v_s,p_s)=-\la\v_s\cdot\nb\v_s+f \quad &{\rm in}\ \ \Om^{T},\\
&\Div v_s=0 \quad &{\rm in}\ \ \Om^{T},\\ &\ \bar n\cdot v_s=0,\ \
\nu\bar n\cdot\D(v_s)\cdot\bar\tau_\al+\ga
v_s\cdot\bar\tau_\al=0\quad &{\rm on}\ \ S_1^{T},\\ &v_s\cdot\bar
n=d\quad &{\rm
on}\  S_2^T ,\\
 & \bar n\cdot\D(v_s)\cdot\bar\tau_\alpha =0,\ \
\alpha=1,2,\quad
 &{\rm on}\ \ S_2^T,\\ & \
v_s|_{t=0}=v(0)\quad &{\rm in}\ \ \Om  \eenn and  \benn
 &h_{st}-\Div\T(h_s,q_s)=-\la(\h_s\cdot\nb\v_s+
 \v_s\cdot\nb\h_s)+g \quad &{\rm in}\ \ \Om^{T} \\
 &\Div h_s=0 \quad &{\rm in}\ \ \Om^{T}\\
&\bar n\cdot h_s=0,\ \ \nu\bar n\cdot\D(h_s)\cdot\bar\tau_\al+\ga
h_s\cdot\bar\tau_\al=0,\ \ \al=1,2\quad &{\rm on}\ \ S_1^{T},\\
&h_{si} = -d_{x_i},\ \ i=1,2,\ \ h_{s3,x_3}= \De' d \quad &{\rm on} \ S_2^T,\\
&h_s|_{t=0}=h(0) \quad &{\rm in}\ \ \Om, \eenn

where $s=1,2$.

\noindent Let
$$
V=v_1-v_2,\quad H=h_1-h_2,\quad P=p_1-p_2,\quad Q=q_1-q_2
$$
Then $V$ and $H$ are solutions to the problems \ben \bal \label{V}
&V_t-\Div\T(V,P)=-\la(\tilde V\cdot\nb\v_1+\v_2\cdot \nb\tilde
V)\\ &\Div V=0\\ &V\cdot\bar n|_{S_1}=0,\ \ \nu\bar
n\cdot\D(V)\cdot\bar\tau_\al+\ga V\cdot\bar\tau_\al|_{S_1}=0,\ \
\al=1,2,\\ &V\cdot\bar n|_{S_2}=d\ \  ,
  \bar n\cdot\D(V)\cdot\bar\tau_\alpha|_{S_2} =0,\ \
\alpha=1,2,\\ &V|_{t=0}=0, \eal \een %
\ben \bal & H_t-\Div\T(H,Q)=-\la(\tilde H\cdot\nb\v_1+\h_2\cdot
\nb\tilde V+\tilde V\cdot\nb\h_1+\v_2\cdot\nb\tilde H)\\
&\Div H=0\\ &H\cdot\bar n|_{S_1}=0,\ \ \nu\bar
n\cdot\D(H)\cdot\bar\tau_\al+\ga H\cdot\bar\tau_\al|_{S_1}=0, \ \
\al=1,2,\\ &
H_i|_{S_2}=-d_{x_i},\ \ i=1,2,\ \ H_{3,x_3}|_{S_2}=\De' d,\\
&H|_{t=0}=0.\eal \een

Assume that $\la\not=0$. In Lemma~\ref{A-lemma}, we proved the
existence of such a constant $\mathcal A$ that for sufficiently
chosen data,
$$ \|h\|_{W^{2,1}_2(\Om^T)} \le \mathcal A \quad {\rm and}
\quad \|v\|_{W^{2,1}_2(\Om^T)} \le \va(\mathcal A). $$ Then for
solutions of (\ref{V}) we have \ben \label{V-est} \bal
&\|V\|_{{\mathcal M}(\Om^{T})}=
\|V\|_{L_{2r}(0,T;W_{\frac{6\eta}{3+\eta}}^1(\Om))}\le c
\|V\|_{W_2^{2,1}(\Om^{T})}\le c\|V\|_{W_{\eta,r}^{2,1}(\Om^{T})}
\\ & \le c\sum_{s=1}^2
\|\v_s\|_{L_{2r}(0;T;W_{\frac{6\eta}{3+\eta}}^1(\Om))}\cdot
\|\tilde V\|_{L_{2r}(0;T;W_{\frac{6\eta}{3+\eta}}^1(\Om))} \le
c(\mathcal A)\|\tilde V\|_{{\mathcal M}(\Om^{T})},\eal \een where
$r\ge2$, $\eta\ge2$ and satisfy either (\ref{imb3}) or
(\ref{imb}).

\noindent Similarly we can prove that \ben \label{H-est}
\|H\|_{{\mathcal M}(\Om^{T})}\le c(\mathcal A) (\|\tilde
V\|_{{\mathcal M}(\Om^{T})}+ \|\tilde H\|_{{\mathcal M}(\Om^{T})})
\een

\noindent Inequalities (\ref{V-est}) and (\ref{H-est}) imply
continuity of mapping $\Phi$. Continuity with respect to $\la$ is
obvious. Hence by the Leray-Schauder fixed point theorem we have
the existence of solutions to problem (\ref{NS}) such that $v\in
W_2^{2,1}(\Om^{T})$, $\nb p\in L_2(\Om^{T})$ and estimates
(\ref{main}) hold.
\end{proof}

Lemma~4.1 and estimates (\ref{h-A}) from Lemma~\ref{A-lemma} imply Theorem 1.

\end{document}